\documentclass[12pt,reqno]{amsart}

\usepackage{mathrsfs}
\usepackage{amsfonts}
\usepackage{amssymb}
\usepackage{amsthm}
\usepackage{amsmath}
\usepackage{latexsym}
\usepackage{tikz}
\usepackage{cite}
\allowdisplaybreaks

\newtheorem{theorem}{Theorem}[section]
\newtheorem{proposition}[theorem]{Proposition}
\newtheorem{lemma}[theorem]{Lemma}

\numberwithin{equation}{section}

\newcommand{\beq}{\begin{equation}}
\newcommand{\eeq}{\end{equation}}

\newcommand{\Rmnum}[1]{\expandafter\@slowromancap\romannumeral #1@}

\newcommand{\p}{\partial}

\newcommand{\ben}{\begin{eqnarray}}
\newcommand{\een}{\end{eqnarray}}
\newcommand{\beno}{\begin{eqnarray*}}
\newcommand{\eeno}{\end{eqnarray*}}

\begin{document}

\title[Global well-posedness of 2D MHD equations]
{Global well-posedness of 2D incompressible MHD equations without magnetic diffusion}

\author[S. Ding]{Shijin Ding}

\address{School of Mathematical Sciences, South China Normal University, Guangzhou, 510631, P.R. China}

\email{dingsj@scnu.edu.cn}

\author[R. Pan]{Ronghua Pan}

\address{School of Mathematics, Georgia Institute of Technology, Atlanta, GA 30332, USA}

\email{panrh@math.gatech.edu}

\author[Y. Zhu]{Yi Zhu}

\address{Department of Mathematics, East China University of Science and Technology, Shanghai 200237,  P.R. China}

\email{zhuyim@ecust.edu.cn}

\date{}
\subjclass[2010]{}

\begin{abstract}
In recent years, the global
existence of classical solutions to the Cauchy problem for 2D incompressible viscous MHD equations without magnetic diffusion has been proved in \cite{Ren,TZhang}, under the assumption that  initial data is close to equilibrium states with nontrivial magnetic field, and the perturbation is small in some suitable spaces, say for instance, the Sobolev spaces with negative exponents. It leads to an interesting open question: Can one establish the global existence of classical solutions without the extra help from Sobolev spaces with negative exponents like its counterparts of ideal MHD ( i.e. without viscosity and magnetic diffusion), and fully dissipative MHD (i.e. with both viscosity and magnetic diffusion)? This paper offers an affirmative answer to this question. In fact, we will establish the existence of a global unique solution for initial perturbations being small  in  $H^2(\mathbb{R}^2)$. The key idea is further exploring the structure of system, using dispersive effects of Alfv\'{e}n waves in the direction which is transversal to the dissipation favorable direction. This motivates our key strategy to treat the wildest nonlinear terms as an artificial linear term. These observations help us to construct some interesting quantities which improves the nonlinearity order for  the wildest terms, and to control them by terms with better properties.
\end{abstract}

\maketitle

\section{introduction}
The dynamics of electrically conducting fluids interacting with magnetic fields can be modelled by the equations of magnetohydrodynamics (MHD).
This paper deals with the following Cauchy problem of 2D incompressible viscous and non-resistive (without magnetic diffusion) MHD equations on the whole space:
\begin{equation}\label{mhd}
  \begin{cases}
    u_t + u \cdot \nabla u - \Delta u + \nabla P = B \cdot \nabla B, \; x \in \mathbb{R}^2, \; t > 0, \\
    B_t + u \cdot \nabla B  = B \cdot \nabla u, \\
    \nabla \cdot u = \nabla \cdot B = 0, \\
    (u,B)|_{t = 0} = (u_0(x), B_0(x)).
  \end{cases}
\end{equation}
Here $u=(u_1(x,t), u_2(x,t))^{\top}$ denotes the velocity field,  $B=(B_1(x,t), B_2(x,t))^{\top}$ the magnetic field, and $P$ the total pressure.
The velocity field obeys the Navier-Stokes equations with Lorentz force. The magnetic field satisfies the non-resistive Maxwell-Faraday equations which describe the Faraday's law of induction.

The MHD equations reflect the basic physics laws governing the motion of electrically conducting fluids such as plasmas, liquid metals and electrolytes. It has played pivotal roles in the study of many phenomena in geophysics, astrophysics, cosmology and engineering (see, e.g., \cite{DL, PF, bis, priest}). Mathematically, it not only shares many crucial features with the Euler or the Navier-Stokes equations but also exhibits many more fascinating properties. These distinctive features make analytic studies a great challenge but offer new opportunities.

We shall briefly recall some of the closely related works about incompressible MHD equations to place our result in a suitable context.  One of the fundamental problems is to establish the global existence of smooth solutions when initial data are small perturbation near equilibrium states with nontrivial background magnetic field. This phenomenon, interpreted in physics, is the stabilization effects of the magnetic field, joint with other dissipative effects such as viscosity and magnetic diffusion if they are taken into accounts. 
It is a long standing open problem that whether or not classical solutions of \eqref{mhd} can develop finite time singularities even in the 2D case. Except with full magnetic diffusion in \eqref{mhd}, the corresponding 2D system possesses a unique global smooth solution (see \cite{ap, st} for initial data in the critical spaces).
For ideal MHD equations, where no viscosity or  magnetic diffusion is included in the equations,  such a problem was solved  in several beautiful papers \cite{BSS, CaiLei, HeXuYu, WeiZ}. The dispersive effect of Alfv\'{e}n waves plays a central role. For the problem of viscous MHD equations without magnetic diffusion, using the dissipative effects of viscosity, global regularity and stability result was first studied in \cite{LinZhang1}, which inspired many further investigations. This result, has been further extended and improved  by several authors via different approaches (see, e.g., \cite{AZ, Deng, HuWang, PanZhouZhu, Ren, Tan, WuWu, WuWuXu, TZhang, xuzhang, zhangting}).
For the inviscid and resistive MHD equations, global $H^1$ weak solution results are obtained in \cite{caowu, leizhou}. We also refer to \cite{zhouzhu, weizhang, zhang2} for more relevant results.
To help readers to gain  a more complete references of current studies on this problem,  we refer to  some other exciting results in
\cite{CaoWuYuan, DongLiWu1, Fefferman1,Fefferman2, FNZ, HuangLi, JNW, JiuZhao2,Yam2, lisun, duzhou, lty, qyyz, jj, jiang, jwz, lwz, zz, liyang, agx, caiyuan, jjz, xjl} and the references therein.

In this paper, we will focus on viscous MHD without resistivity (magnetic diffusion) (\ref{mhd})  in two space dimensions. This system is not merely a combination of the Navier-Stokes on fluid velocity and the transport equations on magnetic field, it is rather
an interactive and integrated system through strong nonlinear coupling between fluid velocity and magnetic fileds.
It is clear that a special solution of \eqref{mhd} is given by the zero velocity field and the nontrivial constant  background magnetic field $e_0 = (0, 1)^T$. The perturbation $(u, b)$ around this equilibrium with $b = B - e_0$ obeys
\begin{equation}\label{mhd0}
  \begin{cases}
    u_t + u \cdot \nabla u - \Delta u + \nabla P = b \cdot \nabla b + \p_2 b, \; x \in \mathbb{R}^2, \; t > 0, \\
    b_t + u \cdot \nabla b  = b \cdot \nabla u + \p_2 u, \\
    \nabla \cdot u = \nabla \cdot b = 0, \\
    (u,b)|_{t = 0} = (u_0(x), b_0(x)).
  \end{cases}
\end{equation}
The absence of magnetic diffusion poses a major challenge in proving the existence of global classical solutions, because it is difficult to control the Lorentz force in the momentum equations. Although at first glance the viscosity seems to provide dissipation to velocity in all directions and no dissipation on magnetic field, the strong nonlinear coupling with magnetic field actually weakened and re-distributed dissipation of viscosity in the sense that the actual dissipation is happening in the favorable direction ($x_2$-direction), but for both velocity and magnetic field. In fact, the linearized system for both velocity $u$ and magnetic field $b$ share the exact same major parts, and can be modeled by the following problem
\begin{equation}\label{linm}
w_{tt}-\Delta w_t-\partial_2^2 w=0, \quad w(0)=w_0, \quad w_t(0)=w_1.\end{equation}
This prevents the  desired bound on $\|\nabla u\|_{L^1_t L_x^\infty}$ from reaching. Experiences gained from past seminar works suggest that one can explore the stabilization effects of the nontrivial background magnetic field to obtain dissipation of magnetic field in the favorable $x_2$-direction, i.e.,
$$ \int_0^t \|\p_2 b(\cdot,\tau)\|_{H^{s-1}(\mathbb{R}^2)}^2 \; d\tau \leq C (\|u_0\|_{H^s(\mathbb{R}^2)}^2 + \|b_0\|_{H^s(\mathbb{R}^2)}^2).$$
for some suitable $s > 0$.
However,  when constructing energy frame to achieve uniform estimates, with the help in mind that higher order nonlinearity improves time decay and smallness of the objects, and with the help of incompressibility, one will soon encounter the following wildest term from nonlinear coupling,
$$\int_0^t \int_{\mathbb{R}^2} \p_2 u_2 |\p_1^s b_2|^2 (x,\tau)\; dx \; d\tau .$$
Recall that we do not have  control on $\|\nabla u\|_{L^1_t L_x^\infty}$, it seems not possible to achieve uniform in time bound on this term unless some helping hands, such as a little faster decay in some terms, say $\p_2 u_2$, are available. One possible ways to achieve this is to ask a little more from the initial data, say, better regularity for low frequency part especially in $x_2$ direction. Indeed, following the explanation of \cite{Ren}, one can solve \eqref{linm} by Fourier transform,
$${\hat w}(t, \xi)=a_+(\xi)e^{t\lambda_+(\xi)}+a_-(\xi)e^{t\lambda_-(\xi)},$$
 $$\lambda_{\pm}(\xi)=-\frac{|\xi|^2}{2}(1\pm\sqrt{1-\frac{4|\xi_2|^2}{|\xi|^4}}),$$
  $$a_{\pm}=\pm\frac{\lambda_{\mp}(\xi){\hat u_0}(\xi)-{\hat u}_1(\xi))}{|\xi|^2\sqrt{1-\frac{4|\xi_2|^2}{|\xi|^4}}}.$$
  For $|\xi|\to+\infty$,
  $$\lambda_-(\xi)\approx -\frac{\xi_2^2}{|\xi|^2}\to 0, \ for \ |\xi|\gg |\xi_2|.$$
  Therefore, one finds the ``best possible" decay to be:
  if $(w_0)\in H^1(\mathbb{R}^2)$, $(|D_2|^{-\sigma}(w_0))\in H^{1+\sigma}(\mathbb{R}^2)$ for $\sigma>0$, then
  $$\|w(t)\|_{L^2(\mathbb{R}^2)}\le C(1+t)^{-\frac{\sigma}{2}}.$$
 Therefore, if one hopes to obtain some time decay in the solution of \eqref{mhd0} to help in the estimates,  it seems to be a natural idea to ask initial data to be in certain anisotropic Sobolev space with negative exponents. \cite{Ren} successfully carried out this idea, and proved the existence of global classical solutions to MHD equations \eqref{mhd0} through anisotropic type estimate with the following initial data
$$ \|(b_0, u_0)\|_{H^8} + \||D|^{-\frac{1}{2} + \epsilon} |D_2|^{-\frac{1}{2} + \epsilon}(b_0, u_0)\|_{L^2} + \||D|^{8} |D_2|^{-\frac{1}{2} + \epsilon} (b_0, u_0)\|_{L^2} \ll 1,$$
for arbitrarily small $\epsilon>0$. Under similar considerations, another interesting work is given by \cite{TZhang}, in which the author made full use of the structure of \eqref{mhd0} and offered a simple proof of this problem with the following initial data
\begin{equation}\label{phase-data}
\|(u_0, b_0)\|_{H^2} + \|e^{-|\xi|^2 t} (\hat b_0, \hat u_0)\|_{L^2([0, \infty); L_{\xi}^1)} \ll 1.
\end{equation}
In fact, the initial assumption $\|(u_0, b_0)\|_{\dot H^{- \sigma_0} \cap \dot H^2} \ll 1$ implies \eqref{phase-data} as long as $\sigma_0 > 0$.

One of the major puzzles in this problem is whether this kind of ``extra" condition on initial data, that is the extra regularity in certain anisotropic Sobolev space with negative exponents, is necessary for the global existence of strong or classical solutions to \eqref{mhd0}. Philosophically speaking, our problem sits in the middle of ideal MHD and fully dissipative MHD where global existence of classical solution are obtained for initial data $(u_0, b_0)$ only be small in Sobolev space $H^s(\mathbb{R}^2)$ for some $s>0$. Therefore, it is natural to ask: {\it Is it possible to establish the global existence of classical solution to \eqref{mhd0} when initial data $(u_0, b_0)$ only be small in Sobolev space $H^s(\mathbb{R}^2)$ (for some $s > 0$) without any negative Sobolev norms requirement?}
 We will address this problem and offer an affirmative answer in this paper. Our main result is stated below.

\begin{theorem}\label{thm}
Consider the Cauchy problem\eqref{mhd0} with initial data $(u_0, b_0) \in H^2({\mathbb{R}^2})$ satisfying $\nabla \cdot u_0 = \nabla \cdot b_0 = 0$. There exists a constant $\epsilon > 0$ such that, if
\begin{equation}\label{initial}
  \|u_0\|_{H^2({\mathbb{R}^2})} + \|b_0\|_{H^2({\mathbb{R}^2})} \leq \epsilon,
\end{equation}
then \eqref{mhd0} has a unique global classical solution $(u, b)$ satisfying, for any $t > 0$,
\begin{equation}\nonumber
\begin{split}
\|u(t)\|_{H^2({\mathbb{R}^2})}^2 + \|b(t)\|_{H^2({\mathbb{R}^2})}^2 + \int_0^t \Big( \|\nabla u(\tau)\|_{H^2({\mathbb{R}^2})}^2 + \|\p_2 b(\tau)\|_{H^1({\mathbb{R}^2})}^2 \Big)\; d\tau \leq& \epsilon.
\end{split}
\end{equation}
\end{theorem}

The proof of Theorem 1.1 is highly non-trivial, or in some sense unexpected. We will briefly explain the main ideas and the pathway of the proof.  As explained earlier, one aims at  bounding $\|u(t)\|_{H^2} + \|b(t)\|_{H^2}$ through the energy estimate, with energy functional,
\begin{equation}\nonumber
\sup_{0 \leq \tau \leq t} (\|u(\tau)\|_{H^2(\mathbb{R}^2)}^2 + \|b(\tau)\|_{H^2(\mathbb{R}^2)}^2) + \int_0^t \Big( \|\nabla u(\tau)\|_{H^2({\mathbb{R}^2})}^2 + \|\p_2 b(\tau)\|_{H^1({\mathbb{R}^2})}^2 \Big)\; d\tau,
\end{equation}
and the major obstacle is to control

$$\int_0^t \int_{\mathbb{R}^2} \p_2 u_2 |\p_1^2 b_2|^2(x,\tau)\; dx \;d\tau,$$
without the help of magnetic diffusion.

To overcome this problem, our first idea is to use the magnetic field equation in \eqref{mhd0} to replace $\p_2 u_2$ by $\p_t b_2 + u \cdot \nabla b_2 - b \cdot \nabla u_2$. This method will transfer the triple nonlinear term to fourth order nonlinear terms and some other good terms that we could control (for more details, please see Proposition \ref{prop1}). Roughly speaking, we have
\begin{equation}\nonumber
\begin{split}
&\int_0^t \int_{\mathbb{R}^2} \p_2 u_2 |\p_1^2 b_2|^2\; dx \;d\tau\\
 =& -\int_0^t \int_{\mathbb{R}^2} b_1 \p_1 u_2 |\p_1^2 b_2|^2\; dx \;d\tau + \text{some good terms} \\
\leq& \Big(\int_0^t \int_{\mathbb{R}^2} |\p_1 u_2|^2 |\p_1^2 b_2|^2 \; dx \; d\tau\Big)^\frac{1}{2}\Big( \int_0^t \int_{\mathbb{R}^2} |b_1|^2 |\p_1^2 b_2|^2 \; dx \; d\tau\Big)^\frac{1}{2} \\
& \qquad \qquad +  \text{some good terms}.
\end{split}
\end{equation}
Then, there comes the most crucial term
\begin{equation}\label{crucial}
\int_0^t \int_{\mathbb{R}^2} |b_1|^2 |\p_1^2 b_2|^2\; dx \;d\tau.
\end{equation}

Noticing the setting of initial data in Theorem \ref{thm}, without any extra help from Sobolev space with negative exponents, it seems hard to reach the control (uniformly in time) for $\|b_1\|_{L_t^2L_x^\infty}$, although it looks better than $\|\nabla u\|_{L^1_t L_x^\infty}$.  More precisely, the fact $ \|b_1\|_{L^\infty} \nleqslant C \|b_1\|_{\dot{H}^1} = C \|\p_2 b\|_{L^2}$  in $\mathbb{R}^2$ implies the fact
$$\int_0^t \int_{\mathbb{R}^2} |b_1|^2 |\p_1^2 b_2|^2 \; dx \; d\tau \nleqslant C \sup_{0 \leq \tau \leq t} \|b\|_{H^2(\mathbb{R}^2)}^2 \int_0^t \|\p_2 b\|_{L^2(\mathbb{R}^2)}^2 \; d\tau .$$
Therefore, one will need to find some other help from the structure of our system.

Our second key idea is to find a ``good" way to explore the dispersive effects of Alfv\'{e}n waves in \eqref{mhd0}. In the seminal work \cite{BSS}, Bardos, Sulem and Sulem used the celebrated Els{\" a}sser variables
 $$Z^{\pm}=u\pm b$$
 to trace down the dispersive effects of Alfv\'{e}n waves, a stabilization effects offered by nontrivial background magnetic field, to establish the global existence of classical solutions to ideal MHD, without viscosity or magnetic diffusion. When both viscosity and magnetic diffusion are in effect and their coefficients are almost equal to each other, \cite{CaiLei, HeXuYu, WeiZ} proved the efficiency of
 Els{\" a}sser variables in this case and established global existence of classical solutions, and the connection to the solutions of ideal MHD in vanishing dissipation limit. However, it seems nearly impossible for Els{\" a}sser variables to be effective in our case, since the ratio between viscosity coefficient and magnetic diffusion coefficient is infinity in \eqref{mhd0}. However, we are not aiming at
 all directions but in $x_1$-direction where dissipation is inactive in principle. This motivated us to try part of Els{\" a}sser variables.

To this end, our strategy is NOT to treat detached quantity in \eqref{crucial} as individual part. Which means we will treat it as a whole and find some "good" equivalent forms for it. We define
$$\Omega_{\pm} \triangleq  u_1 \pm b_1, $$
\begin{equation}\label{G}
G_{\pm} \triangleq |\Omega_{\pm}|^2 \mp \p_t \int_{-\infty}^{x_2} |\Omega_{\pm}|^2 \; dy.
\end{equation}
Such strategy takes advantage of the idea that regards the nonlinear terms as an artificial linear term, and study its evolution.
We remark that, the variable $\Omega_{\pm}$ are actually the ``good" variables, but it causes errors due to the mismatch between the viscosity coefficient and magnetic diffusion coefficient in \eqref{mhd0}. Therefore, one would need to construct the correction parts to cancel out majority of the error. This leads to the introduction of quantities  $G_{\pm}$.

Now, we are able to find out that
\begin{equation}\nonumber
\begin{split}
& \int_0^t \int_{\mathbb{R}^2} |b_1|^2 |\p_1^2 b_2|^2\; dx \;d\tau \\
= & \int_0^t \int_{\mathbb{R}^2} \big|\frac{\Omega_+ - \Omega - }{2}\big|^2 |\p_1^2 b_2|^2\; dx \;d\tau \\
\leq & \int_0^t \int_{\mathbb{R}^2} (\Omega_{+}^2 + \Omega_{-}^2) |\p_1^2 b_2|^2\; dx \;d\tau \\
\leq & \; C\int_0^t \int_{\mathbb{R}^2} (G_{+}  + G_{-})|\p_1^2 b_2|^2\; dx \;d\tau + \text{some good terms},
\end{split}
\end{equation}
 where, $\int_0^t \int_{\mathbb{R}^2} G_{\pm} |\p_1^2 b_2|^2\; dx \;d\tau$ would be estimated through  the following identity
\begin{equation}\label{G2}
  G_{\pm}  = 2 \int_{-\infty}^{x_2} (\pm u \cdot \nabla \Omega_{\pm} \mp b \cdot \nabla \Omega_{\pm} \mp \Delta u_1 \pm \p_1 P) \Omega_{\pm} \; dy.
\end{equation}
For more details, we refer to Proposition \ref{prop0}.

From the discussion above, it is then clear that one needs to define the following energy framework:
\begin{equation}\label{energy}
  \begin{split}
    E_0(t) =& \sup_{0 \leq \tau \leq t} (\|u(\cdot,\tau)\|_{H^2(\mathbb{R}^2)}^2 + \|b(\cdot,\tau)\|_{H^2(\mathbb{R}^2)}^2) + \int_0^t \|\nabla u(\cdot,\tau)\|_{H^2(\mathbb{R}^2)}^2 \; d\tau, \\
    E_1(t) =& \int_0^t \|\p_2 b(\cdot,\tau)\|_{H^1(\mathbb{R}^2)}^2 \; d\tau, \\
    A_{\pm}(t) =& \int_0^t \int_{\mathbb{R}^2} |\Omega_{\pm}|^2 |\p_1^2 b_2|^2(x,\tau)\; dx \; d\tau,
  \end{split}
\end{equation}
and the total energy $ \mathscr{E}(t)$ is defined as follows,
$$\mathscr{E}(t) = E_0(t) + E_1(t) + A_-(t) + A_+(t) . $$

In summary, we not only proved the global existence of classical solutions to \eqref{mhd0} without using any additional condition on the initial data from Sobolev space with negative components, but also discovered an effective approach to take advantage of dispersion of Alfv\'{e}n waves in the case when viscosity and magnetic diffusion are not well-matched. The arrange of the rest of paper are follows. In Section 2, we explicit some preliminaries on the estimates of $\Omega_{\pm}$ and $G_{\pm}$, and how they will help with the estimate of  the crucial term \eqref{crucial}. We will then establish the desired energy estimates in Section 3, and complete the proof of Theorem 1.1 in Section 4.

\vskip .3in
\section{Preliminaries}\label{prelim}
In order to capture the inherent structure of the system, we have introduced suitable functions to obtain the directional decay of solutions previously. In this section, we present two propositions to derive the energy estimate for two crucial terms:
$$ \int_0^t \int_{\mathbb{R}^2}  G_{\pm} |\p_1^2 b_2|^2 \; dx \; d\tau \quad and \quad \int_0^t \int_{\mathbb{R}^2} \p_2 u_2 |\p_1^2 b_2|^2 \; dx \; d\tau,$$
which is critical in the estimate for $A_{\pm}(t)$ we defined in \eqref{energy}. Moveover, $A_{\pm}(t)$ is the core portion throughout our paper and its estimate plays key role in our whole proof.

\begin{proposition}\label{prop0}
For $G_{\pm}$ defined in \eqref{G}, \eqref{G2} and the energy framework in \eqref{energy},  the following estimates hold
\begin{equation}\nonumber
  \begin{split}
   \int_0^t \int_{\mathbb{R}^2}  G_{\pm} |\p_1^2 b_2|^2 \; dx \; d\tau = & \int_0^t \int_{\mathbb{R}^2}  \big(|\Omega_{\pm}|^2 \mp \p_t \int_{-\infty}^{x_2} |\Omega_{\pm}|^2 \; dy\big) |\p_1^2 b_2|^2 \; dx \; d\tau \\
    \leq &  \; C (E_0 + E_0^\frac{3}{2}) (E_0 + E_1) + CE_0^\frac{1}{2} A_{\pm} + C E_0 A_{\pm}^\frac{1}{2}.
  \end{split}
\end{equation}
\end{proposition}
\begin{proof}
Using MHD equations \eqref{mhd0}, direct computation shows 
\begin{equation}\label{omega}
  \p_2 \Omega_{\pm} \mp \p_t \Omega_{\pm} = \pm u \cdot \nabla \Omega_{\pm} - b \cdot \nabla \Omega_{\pm} \mp \Delta u_1 \pm \p_1 P.
\end{equation}
Multiply \eqref{omega} by $2 \Omega_{\pm}$ and integrate it over $(-\infty, x_2]$, one has
\begin{equation}\label{omega1}
|\Omega_{\pm}|^2 \mp \p_t \int_{-\infty}^{x_2} |\Omega_{\pm}|^2 \; dy = 2 \int_{\infty}^{x_2} \Big(\pm u \cdot \nabla \Omega_{\pm} - b \cdot \nabla \Omega_{\pm} \mp \Delta u_1 \pm \p_1 P\Big)  \Omega_{\pm}\; dy.
\end{equation}
Then, multiply \eqref{omega1} by $|\p_1^2 b_2|^2$ and integrate it over $\mathbb{R}^2 \times [0, t]$, it holds that
\begin{equation}\nonumber
  \begin{split}
    \int_0^t \int_{\mathbb{R}^2} \big( |\Omega_{\pm}|^2 \mp \p_t \int_{-\infty}^{x_2} |\Omega_{\pm}|^2 \; dy \big)|\p_1^2 b_2|^2 \; dx \; d\tau = \sum_{i = 1}^3 I_i^{\pm},
  \end{split}
\end{equation}
where,
\begin{equation}\nonumber
  \begin{split}
    I_1^{\pm} =& \pm 2 \int_0^t \int_{\mathbb{R}^2} \Big( \int_{-\infty}^{x_2}( u \cdot \nabla \Omega_{\pm} \mp b \cdot \nabla \Omega_{\pm}) \Omega_{\pm} \; dy \Big) |\p_1^2 b_2|^2\; dx \; d\tau, \\
    I_2^{\pm} =& \mp 2 \int_0^t \int_{\mathbb{R}^2} \Big( \int_{-\infty}^{x_2} \Delta u_1  \Omega_{\pm} \; dy  \Big)|\p_1^2 b_2|^2 \; dx \; d\tau, \\
    I_3^{\pm} =& \pm 2 \int_0^t \int_{\mathbb{R}^2} \Big( \int_{-\infty}^{x_2} \p_1 P \Omega_{\pm} \; dy \Big) |\p_1^2 b_2|^2 \; dx \; d\tau.
  \end{split}
\end{equation}
By using integration by parts, we can divide the first term $I_1^{\pm}$ into three parts
\begin{equation}\nonumber
  \begin{split}
    I_1^{\pm} =& \pm 2\int_0^t \int_{\mathbb{R}^2} \Big( \int_{-\infty}^{x_2} (u_1 \mp b_1) \p_1 \Omega_{\pm} \Omega_{\pm}  \; dy \Big) |\p_1^2 b_2|^2 \; dx \; d\tau \\
    & \mp \int_0^t \int_{\mathbb{R}^2}  (u_2 \mp b_2) |\Omega_{\pm}|^2 |\p_1^2 b_2|^2 \; dx \; d\tau  \\
    & \pm \int_0^t \int_{\mathbb{R}^2} \Big(  \int_{-\infty}^{x_2} \p_2 (u_2 \mp b_2) |\Omega_{\pm}|^2 \; dy \Big) |\p_1^2 b_2|^2 \; dx \; d\tau .
\end{split}
\end{equation}
H\"{o}lder's inequality then implies
\begin{equation}\nonumber
\begin{split}
 I_1^{\pm}   \leq & \; C \sup_{0 \leq \tau \leq t} \|\p_1^2 b_2\|_{L^2}^2 \int_0^t \Big \|\int_{-\infty}^{x_2} (u_1 \mp b_1) \p_1 \Omega_{\pm} \Omega_{\pm}  \; dy\Big \|_{L^\infty} \; d\tau \\
    & + C \sup_{0 \leq \tau \leq t} \|(u_2 \mp b_2)\|_{L^\infty} \int_0^t \int_{\mathbb{R}^2} |\Omega_{\pm}|^2 |\p_1^2 b_2|^2 \; dx \; d\tau \\
    & + C \sup_{0 \leq \tau \leq t} \|\p_1^2 b_2\|_{L^2}^2 \int_0^t \Big \| \int_{-\infty}^{x_2} \p_2 (u_2 \mp b_2) |\Omega_{\pm}|^2 \; dy \Big \|_{L^\infty} \; d\tau.
  \end{split}
\end{equation}
Apply the anisotropic type inequalities to the nonlinear terms, the first line above can be controlled by 
$$ C \sup_{0 \leq \tau \leq t}\|\p_1^2 b_2\|_{L^2}^2 \int_0^t \|(u_1 \mp b_1)\|_{L^\infty} \|\p_1 \Omega_{\pm}\|_{L_{x_1}^\infty L_{x_2}^2} \|\Omega_{\pm}\|_{L_{x_1}^\infty L_{x_2}^2} \; d\tau,$$
and the third line in above inequality can be bounded by
$$ C \sup_{0 \leq \tau \leq t} \|\p_1^2 b_2\|_{L^2}^2 \int_0^t \|\p_2 (u_2 \mp b_2)\|_{L_{x_1}^\infty L_{x_2}^2} \|\Omega_{\pm}\|_{L_{x_1}^\infty L_{x_2}^2} \|\Omega_{\pm}\|_{L^\infty} \; d\tau.$$
Consequently, $I_1^{\pm}(t)$ can be controlled by the following terms
\begin{equation}\nonumber
\begin{split}
   & C \sup_{0 \leq \tau \leq t}\|\p_1^2 b_2\|_{L^2}^2 \int_0^t \|(u_1 \mp b_1)\|_{L^2}^\frac{1}{2} \|\nabla^2 (u_1 \mp b_1)\|_{L^2}^\frac{1}{2} \\
   &\qquad \qquad\qquad \|\nabla \p_1 \Omega_{\pm}\|_{L^2}^\frac{1}{2} \|\Omega_{\pm}\|_{L^2}^\frac{1}{2} \|\p_1 \Omega_{\pm}\|_{L^2}  \; d\tau \\
   + \;&C \sup_{0 \leq \tau \leq t} \|(u_2 \mp b_2)\|_{H^2} \int_0^t \int_{\mathbb{R}^2} |\Omega_{\pm}|^2 |\p_1^2 b_2|^2 \; dx \; d\tau \\
   + \;&C \sup_{0 \leq \tau \leq t} \|\p_1^2 b_2\|_{L^2}^2 \int_0^t  \Big\{ \|\p_2 (u_2 \mp b_2)\|_{L^2}^\frac{1}{2} \|\p_1 \p_2 (u_2 \mp b_2)\|_{L^2}^\frac{1}{2} \\
    &  \qquad \qquad  \cdot  \|\Omega_{\pm}\|_{L^2}^\frac{1}{2}\|\p_1 \Omega_{\pm}\|_{L^2}^\frac{1}{2} \|\Omega_{\pm}\|_{L^2}^\frac{1}{2} \|\nabla^2 \Omega_{\pm}\|_{L^2}^\frac{1}{2} \Big\}\;  d\tau .
\end{split}
\end{equation}
Moreover, we can write
\begin{equation}\nonumber
\begin{split}
I_1^{\pm}(t) \leq  & \;C \sup_{0 \leq \tau \leq t}\|(u_1 \mp b_1)\|_{L^2}^\frac{1}{2}\|\Omega_{\pm}\|_{L^2}^\frac{1}{2} \|\p_1^2 b_2\|_{L^2}^2 \int_0^t \|\nabla^2 (u_1 \mp b_1)\|_{L^2}^\frac{1}{2} \\
& \qquad\qquad \|\nabla \p_1 \Omega_{\pm}\|_{L^2}^\frac{1}{2}\|\p_1 \Omega_{\pm}\|_{L^2} \; d\tau \\
   + & \;C\sup_{0 \leq \tau \leq t} \|(u_2 \mp b_2)\|_{H^2} \int_0^t \int_{\mathbb{R}^2} |\Omega_{\pm}|^2 |\p_1^2 b_2|^2 \; dx \; d\tau \\
   + & \;C \sup_{0 \leq \tau \leq t} \|\Omega_{\pm}\|_{L^2}\|\p_1^2 b_2\|_{L^2}^2 \int_0^t \|\p_2 (u_2 \mp b_2)\|_{L^2}^\frac{1}{2} \|\p_1 \p_2 (u_2 \mp b_2)\|_{L^2}^\frac{1}{2} \\
    &  \qquad \qquad\qquad  \cdot \|\p_1 \Omega_{\pm}\|_{L^2}^\frac{1}{2} \|\nabla^2 \Omega_{\pm}\|_{L^2}^\frac{1}{2} \; d\tau .
  \end{split}
\end{equation}
We then arrive at the  following estimate for $I_1^{\pm}$
\begin{equation}\label{I1}
 I_1^{\pm}  \leq  \;  CE_0^\frac{3}{2}(E_0 + E_1) + CE_0^\frac{1}{2} A_{\pm}.
\end{equation}
For the next term $I_2^{\pm}$, using the divergence free condition $\nabla \cdot u = 0$, we can split it into three parts.
\begin{equation}\nonumber
  \begin{split}
    I_2^{\pm} (t) =&  \mp 2 \int_0^t \int_{\mathbb{R}^2} \Big( \int_{-\infty}^{x_2} \p_2(- \p_1 u_2 + \p_2 u_1)  \Omega_{\pm} \; dy \Big)|\p_1^2 b_2|^2 \; dx \; d\tau, \\
     = & \mp 2 \int_0^t \int_{\mathbb{R}^2}(- \p_1 u_2 + \p_2 u_1)  \Omega_{\pm}  |\p_1^2 b_2|^2 \; dx \; d\tau \\
     & \pm 2 \int_0^t \int_{\mathbb{R}^2} \Big( \int_{-\infty}^{x_2} (- \p_1 u_2 + \p_2 u_1)  \p_2 \Omega_{\pm} \; dy \Big) |\p_1^2 b_2|^2 \; dx \; d\tau. \\
\end{split}
\end{equation}
Using H\"{o}lder's inequality and some anisotropic type inequalities, it's obviously that
\begin{equation}\label{I2}
\begin{split}
I_2^{\pm} \leq & \; C \Big(\int_0^t \int_{\mathbb{R}^2} |\nabla u|^2 |\p_1^2 b_2|^2 \; dx \; d\tau\Big)^\frac{1}{2} \Big(\int_0^t \int_{\mathbb{R}^2} |\Omega_{\pm}|^2 |\p_1^2 b_2|^2 \; dx \; d\tau\Big)^\frac{1}{2} \\
     & + C \sup_{0 \leq \tau \leq t} \|\p_1^2 b_2\|_{L^2}^2 \int_0^t \|\nabla u\|_{L_{x_1}^\infty L_{x_2}^2}   \|\p_2 \Omega_{\pm} \|_{L_{x_1}^\infty L_{x_2}^2}  \; d\tau \\
\leq & \; C \Big(\int_0^t \int_{\mathbb{R}^2} |\nabla u|^2 |\p_1^2 b_2|^2 \; dx \; d\tau\Big)^\frac{1}{2} \Big(\int_0^t \int_{\mathbb{R}^2} |\Omega_{\pm}|^2 |\p_1^2 b_2|^2 \; dx \; d\tau\Big)^\frac{1}{2} \\
     & + C \sup_{0 \leq \tau \leq t} \|\p_1^2 b_2\|_{L^2}^2 \int_0^t \|\nabla u\|_{H^2}   \|\p_2 b \|_{H^1}  \; d\tau \\
     \leq & \; C E_0 A_{\pm}^\frac{1}{2} + E_0(E_0 + E_1).
  \end{split}
\end{equation}
Now, let's focus on the last term $I_3^{\pm}$, notice that simple calculation implies the equality
$ P = (-\Delta)^{-1} \p_i \p_j(u_i u_j - b_i b_j)$, we can rewrite $I_3^{\pm}$ as follows.
\begin{equation}\nonumber
  \begin{split}
   I_3^{\pm} = &    \pm 2 \int_0^t \int_{\mathbb{R}^2} \Big( \int_{-\infty}^{x_2} \p_1 P \Omega_{\pm} \; dy \Big) |\p_1^2 b_2|^2 \; dx \; d\tau \\
     = & \pm 2 \int_0^t \int_{\mathbb{R}^2} \Big( \int_{-\infty}^{x_2} \Big \{\sum_{ i,j=1,2;2,1;2,2}(-\Delta)^{-1} \p_1 \p_i \p_j(u_i u_j - b_i b_j) \\
     & + (-\Delta)^{-1} \p_1^3(|u_1|^2 - |b_1|^2) \Big\}\Omega_{\pm}\; dy \Big) |\p_1^2 b_2|^2 \; dx \; d\tau,
\end{split}
\end{equation}
which can be controlled directly by
\begin{equation}\nonumber
\begin{split}
& \sup_{0\leq \tau \leq t} \|\p_1^2 b_2\|_{L^2}^2\int_0^t \Big(\sum_{i,j=1,2;2,1;2,2}\|(-\Delta)^{-1} \p_1 \p_i \p_j(u_i u_j - b_i b_j)\|_{L_{x_1}^\infty L_{x_2}^2}  \\
     &+\|(-\Delta)^{-1} \p_1^3(|u_1|^2 - |b_1|^2)\|_{L_{x_1}^\infty L_{x_2}^2} \Big)\|\Omega_{\pm}\|_{L_{x_1}^\infty L_{x_2}^2} \; d\tau.\\
\end{split}
\end{equation}
Therefore, $I_3^{\pm}$ can be bounded like

\begin{equation}\label{I3-temp}
\begin{split}
 I_3^{\pm}     \leq & \sup_{0\leq \tau \leq t} \|\p_1^2 b_2\|_{L^2}^2\int_0^t \Big(\sum_{i,j=1,2;2,1;2,2}\big(\|u_i \p_2 u_j\|_{L^2}^\frac{1}{2} \|\p_1(u_i \p_2 u_j)\|_{L^2}^\frac{1}{2}  \\
 &\quad + \|b_i \p_2 b_j\|_{L^2}^\frac{1}{2}\|\p_1(b_i \p_2 b_j)\|_{L^2}^\frac{1}{2}\big)
     +\|u_1 \p_2 u_2\|_{L^2}^\frac{1}{2}\|\p_1(u_1 \p_2 u_2)\|_{L^2}^\frac{1}{2}  \\
     & \qquad \qquad + \|b_1 \p_2 b_2\|_{L^2}^\frac{1}{2}\|\p_1(b_1 \p_2 b_2)\|_{L^2}^\frac{1}{2} \Big) \|\Omega_{\pm}\|_{L^2}^\frac{1}{2}\|\p_1 \Omega_{\pm}\|_{L^2}^\frac{1}{2} \; d\tau.
 \end{split}
\end{equation}
According to anisotropic type estimate, the following inequalities hold when the right-hand sides are all bounded.
\begin{equation}\nonumber
  \begin{split}
    \|f \p_2 g\|_{L^2} \leq& C \|f\|_{L_{x_1}^2 L_{x_2}^\infty} \|\p_2 g\|_{L_{x_1}^\infty L_{x_2}^2} \\
    \leq& C \|f\|_{L^2}^\frac{1}{2} \|\p_2 f\|_{L^2}^\frac{1}{2} \|\p_2 g\|_{L^2}^\frac{1}{2} \|\p_1 \p_2 g\|_{L^2}^\frac{1}{2}, \\
    \|\p_1 (f \p_2 g)\|_{L^2} \leq& C \|\p_1f \p_2 g\|_{L^2} + \|f \p_1\p_2 g\|_{L^2} \\
    \leq& C \|\p_1 f\|_{L^2}^\frac{1}{2} \|\p_1 \p_2 f\|_{L^2}^\frac{1}{2} \|\p_2 g\|_{L^2}^\frac{1}{2} \|\p_1 \p_2 g\|_{L^2}^\frac{1}{2} \\
    & + \|f\|_{L^\infty} \|\p_1\p_2g\|_{L^2} \\
    \leq& C \|\p_1 f\|_{L^2}^\frac{1}{2} \|\p_1 \p_2 f\|_{L^2}^\frac{1}{2} \|\p_2 g\|_{L^2}^\frac{1}{2} \|\p_1 \p_2 g\|_{L^2}^\frac{1}{2} \\
    & + \|f\|_{L^2}^\frac{1}{4} \|\p_1f\|_{L^2}^\frac{1}{4} \|\p_2 f\|_{L^2}^\frac{1}{4} \|\p_1\p_2 f\|_{L^2}^\frac{1}{4} \|\p_1\p_2g\|_{L^2}.
  \end{split}
\end{equation}
Using above inequalities, \eqref{I3-temp} then becomes
\begin{equation}\label{I3}
\begin{split}
    I_3^{\pm} \leq & \sup_{0\leq \tau \leq t} (\|u\|_{H^2}^\frac{1}{2} + \|b\|_{H^2}^\frac{1}{2})\|\Omega_{\pm}\|_{L^2}^\frac{1}{2} \|\p_1^2 b_2\|_{L^2}^2\\
    & \cdot \int_0^t (\|\nabla u\|_{H^2}^\frac{3}{2} + \|\p_2 b\|_{H^1}^\frac{3}{2})\|\p_1 \Omega_{\pm}\|_{L^2}^\frac{1}{2} \; d\tau \\
     \leq & \; C E_0^\frac{3}{2} (E_0 + E_1).
  \end{split}
\end{equation}
Taking all estimate for $I_1^{\pm}$ through $I_3^{\pm}$ into account, namely \eqref{I1}, \eqref{I2} and \eqref{I3}, we complete the proof of this Proposition.

\end{proof}

The next Proposition plays a decisive role in proving our main theorem. It shows the inherent structure of system concerned can help improve the nonlinear order of some wild terms.
As a result, the wildest nonlinear term will transfer to some other good terms.

\begin{proposition}\label{prop1}

For $A_{mp}$, $E_0$ and $E_1$ defined in \eqref{energy}, it holds that
\begin{equation}\nonumber
\begin{split}
\int_0^t \int_{\mathbb{R}^2} \p_2 u_2 |\p_1^2 b_2|^2 \; dx \; d\tau \leq & \; CE_0 ( A_+ + A_-)^\frac{1}{2} \\
& + C (E_0^\frac{1}{2}+ E_0 + E_0^\frac{3}{2})(E_0 + E_1).
\end{split}
\end{equation}
\end{proposition}
\begin{proof}
Using the second equation of the  system \eqref{mhd0}, we have
\begin{equation}\nonumber
\begin{split}
  \int_{\mathbb{R}^2} \p_2 u_2 |\p_1^2 b_2|^2 \;dx =& \int_{\mathbb{R}^2} (\p_t b_2 + u \cdot \nabla b_2 - b \cdot \nabla u_2) |\p_1^2 b_2|^2 \; dx \\
  = & \; J_1 + J_2 + J_3,
\end{split}
\end{equation}
where,
\begin{equation}\nonumber
  \begin{split}
    J_1 =& -\int_{\mathbb{R}^2} b_1 \p_1 u_2 |\p_1^2 b_2|^2 \; dx, \\
    J_2 =& -\int_{\mathbb{R}^2} b_2 \p_2 u_2 |\p_1^2 b_2|^2 \; dx, \\
    J_3 =& \int_{\mathbb{R}^2} (\p_t b_2 + u \cdot \nabla b_2) |\p_1^2 b_2|^2 \; dx.
  \end{split}
\end{equation}
Applying H\"{o}lder's inequality and notice that $b_1 = \frac{\Omega_+ - \Omega_- }{2}$,  we have
\begin{equation}\label{J1}
  \begin{split}
   & \int_0^t J_1 \;d\tau \\
    \leq & \Big(\int_0^t \int_{\mathbb{R}^2} |\p_1 u_2|^2 |\p_1^2 b_2|^2 \; dx \; d\tau\Big)^\frac{1}{2}\Big( \int_0^t \int_{\mathbb{R}^2} |b_1|^2 |\p_1^2 b_2|^2 \; dx \; d\tau\Big)^\frac{1}{2} \\
    \leq & C E_0 ( A_+ + A_-)^\frac{1}{2}.
  \end{split}
\end{equation}
For the next term $J_2$, we can split it into two parts and use the evolution of $b_2$ again. 
\begin{equation}\nonumber
  \begin{split}
    J_2 =& -\int_{\mathbb{R}^2} b_2 \p_2 u_2 |\p_1^2 b_2|^2 \; dx \\
        =& -\int_{\mathbb{R}^2} b_2 |\p_1^2 b_2|^2 ( - b \cdot \nabla u_2 + \p_t b_2 + u \cdot \nabla b_2) \; dx \\
        =& \; J_{2,1} + J_{2,2}.
  \end{split}
\end{equation}
Like before, we can use the anisotropic type inequalities to control  $J_{2,1}$ as follows
\begin{equation}\nonumber
  \begin{split}
    J_{2,1} =& \int_{\mathbb{R}^2} b_2 |\p_1^2 b_2|^2  b \cdot \nabla u_2 \; dx \\
    \leq& \; C \|b\|_{L^\infty}^2 \|\nabla u\|_{L^\infty} \|b\|_{H^2}^2 \\
    \leq& \; C \|b\|_{L^2}^\frac{1}{2} \|\p_1 b\|_{L^2}^\frac{1}{2} \|\p_2 b\|_{L^2}^\frac{1}{2} \|\p_1 \p_2 b\|_{L^2}^\frac{1}{2} \|\nabla u\|_{H^2} \|b\|_{H^2}^2\\
    \leq& \; C \|b\|_{H^2}^3 \|\p_2 b\|_{H^1} \|\nabla u\|_{H^2},
  \end{split}
\end{equation}
which then implies
\begin{equation}\label{J21}
  \int_0^t J_{2,1} \;d\tau  \leq  E_0^{\frac{3}{2}} (E_0 + E_1).
\end{equation}
The second part $J_{2,2}$ needs more attention. Using integration by parts, one has
\begin{equation}\nonumber
  \begin{split}
 & \;\quad J_{2,2} \\
  &= -\frac{1}{2}\frac{d}{dt}\int_{\mathbb{R}^2} |b_2|^2 |\p_1^2 b_2|^2 \;dx +\frac{1}{2} \int_{\mathbb{R}^2} |b_2|^2 \frac{d}{dt} |\p_1^2 b_2|^2 - u \cdot \nabla |b_2|^2|\p_1^2 b_2|^2 \; dx \\
    & = -\frac{1}{2}\frac{d}{dt}\int_{\mathbb{R}^2} |b_2|^2 |\p_1^2 b_2|^2 \;dx + \int_{\mathbb{R}^2} |b_2|^2\p_1^2 b_2 \p_1^2(\p_2 u_2 + b \cdot \nabla u_2 -u \cdot \nabla b_2)\\
     & \qquad \qquad - \frac{1}{2}u \cdot \nabla |b_2|^2|\p_1^2 b_2|^2 \; dx \\
    & = -\frac{1}{2}\frac{d}{dt}\int_{\mathbb{R}^2} |b_2|^2 |\p_1^2 b_2|^2\;dx + \int_{\mathbb{R}^2} |b_2|^2\p_1^2 b_2 \\
    & \qquad \big\{\p_1^2(\p_2 u_2 + b \cdot \nabla u_2) -[\p_1^2, u \cdot \nabla] b_2\big\} \; dx .
 \end{split}
 \end{equation}
We now estimate them  term by term.
Applying anisotropic type inequalities, it yields that
\begin{equation}\nonumber
  \begin{split}
    &\int_{\mathbb{R}^2} |b_2|^2 \p_1^2 b_2 \p_1^2 \p_2 u_2 \; dx \\
    \leq& C \|b_2\|_{L^\infty}^2 \|\p_1^2 b_2\|_{L^2} \|\p_1^2 \p_2 u_2\|_{L^2} \\
    \leq& C \|b_2\|_{L^2}^\frac{1}{2}\|\p_1 b_2\|_{L^2}^\frac{1}{2}\|\p_2 b_2\|_{L^2}^\frac{1}{2}\|\p_1 \p_2 b_2\|_{L^2}^\frac{1}{2} \|\p_1^2 b_2\|_{L^2} \|\p_1^2 \p_2 u_2\|_{L^2}\\
    \leq& C \|b\|_{H^2}^2 \|\p_2 b\|_{H^1} \|\nabla u\|_{H^2},\\
\end{split}
\end{equation}
and,
\begin{equation}\nonumber
\begin{split}
    &\int_{\mathbb{R}^2} |b_2|^2 \p_1^2 b_2 \p_1^2 (b \cdot \nabla u_2) \; dx \\
    \leq& C \|b_2\|_{L^\infty}^2 \|\p_1^2 b_2\|_{L^2} (\|\p_1^2 b\|_{L^2} \|\nabla u_2\|_{L^\infty} + \|\p_1 b\|_{L_{x_1}^2 L_{x_2}^\infty} \|\p_1 \nabla u_2\|_{L_{x_1}^\infty L_{x_2}^2} \\
    & + \|b\|_{L^\infty}\|\p_1^2 \nabla u_2\|_{L^2}) \\
    \leq& C \|b\|_{H^2}^3 \|\p_2 b\|_{H^1} \|\nabla u\|_{H^2},\\
  \end{split}
\end{equation}
moreover,
\begin{equation}\nonumber
\begin{split}
  &-\int_{\mathbb{R}^2} |b_2|^2\p_1^2 b_2 [\p_1^2, u \cdot \nabla] b_2 \; dx\\
  \leq & C \|b_2\|_{L^\infty}^2 \|\p_1^2 b_2\|_{L^2} (\|\p_1^2 u\|_{L^2} \|b_2\|_{L^\infty} + \|\p_1 u\|_{L^\infty}\|\p_1 \nabla b_2\|_{L^2})\\
  \leq& C \|b\|_{H^2}^3 \|\p_2 b\|_{H^1} \|\nabla u\|_{H^2}.
\end{split}
\end{equation}
The first part in $J_{2,2}$ behaves well and can be bounded easily. Combing the above three parts together, we finally  derive the estimate for $J_{2,2}$.
\begin{equation}\nonumber
\begin{split}
 J_{2,2} \leq -\frac{1}{2}\frac{d}{dt}\int_{\mathbb{R}^2} |b_2|^2 |\p_1^2 b_2|^2\;dx + C (\|b\|_{H^2}^2 + \|b\|_{H^2}^3) \|\p_2 b\|_{H^1} \|\nabla u\|_{H^2}.
  \end{split}
\end{equation}
Combing the estimate for $J_{2,1}$ and $J_{2,2}$, we then arrive at the time integral estimate for $J_2$
\begin{equation}\label{J2}
  \int_0^t J_2 \; d\tau \leq  C E_0^2 + C (E_0 + E_0^\frac{3}{2})(E_0 + E_1).
\end{equation}
Next, we turn to deal with $J_3$. To this end,  we rewrite $J_3$ as follows.

\begin{equation}\nonumber
  \begin{split}
    J_3
    =& \; \frac{d}{dt}\int_{\mathbb{R}^2} b_2 |\p_1^2 b_2|^2 \; dx - \int_{\mathbb{R}^2}  b_2 \frac{d}{dt} |\p_1^2 b_2|^2 + u \cdot \nabla b_2 |\p_1^2 b_2|^2 \; dx \\
    = & \; \frac{d}{dt}\int_{\mathbb{R}^2} b_2 |\p_1^2 b_2|^2 \; dx + \int_{\mathbb{R}^2}  2 b_2 \p_1^2 b_2 \p_1^2 (u\cdot \nabla b_2 - b \cdot \nabla u_2 - \p_2 u_2) \\
    &\qquad + u \cdot \nabla b_2 |\p_1^2 b_2|^2 \; dx \\
    = & \; \frac{d}{dt}\int_{\mathbb{R}^2} b_2 |\p_1^2 b_2|^2 \; dx + \int_{\mathbb{R}^2} 2 b_2 \p_1^2 b_2 (\p_1^2 u\cdot \nabla b_2 + 2\p_1 u \cdot \nabla \p_1 b_2) \\
    & \qquad \qquad - 2b_2 \p_1^2 b_2 \p_1^2(b \cdot \nabla u_2 + \p_2 u_2)\; dx.
\end{split}
\end{equation}
More specifically, $J_3$ consists of the following terms

\begin{equation}\nonumber
\begin{split}
& \; \frac{d}{dt}\int_{\mathbb{R}^2} b_2 |\p_1^2 b_2|^2 \; dx + \int_{\mathbb{R}^2}2  b_2 \p_1^2 b_2 (\p_1^2 u\cdot \nabla b_2 + 2\p_1 u_2 \p_2 \p_1 b_2 + 2\p_1 u_1 \p_1^2 b_2)\\
    & - 2b_2 \p_1^2 b_2 \p_1^2 b_1 \p_1 u_2 - 2b_2 \p_1^2 b_2 \p_1^2 b_2 \p_2 u_2 - 4 b_2 \p_1^2 b_2 \p_1 b \cdot \nabla \p_1 u_2 - 2 b_2 \p_1^2 b_2  b \cdot \nabla \p_1^2 u_2 \\
    & +2  \p_2 b_2 \p_1^2 b_2 \p_1^2 u_2 -2 \p_1 b_2 \p_1 \p_2 b_2 \p_1^2 u_2-2 b_2 \p_1 \p_2 b_2 \p_1^3 u_2 \; dx.
\end{split}
\end{equation}
Similar to the treatment for $J_{2,2}$, one has
\begin{equation}\nonumber
\begin{split}
 J_3    \leq  & \; \frac{d}{dt}\int_{\mathbb{R}^2} b_2 |\p_1^2 b_2|^2 \; dx - 6 \int_{\mathbb{R}^2} b_2 \p_2 u_2 |\p_1^2 b_2|^2 \; dx \\
 & \qquad  + C (\|b\|_{H^2} + \|b\|_{H^2}^2) \|\p_2 b\|_{H^1} \|\nabla u\|_{H^2}.
  \end{split}
\end{equation}
Notice that the second term on the right hand side in above inequality is
$$\int_{\mathbb{R}^2} b_2 \p_2 u_2 |\p_1^2 b_2|^2 dx,$$
which is just the term $J_2$ and has been estimated before in \eqref{J2}. We can then obtain the control for $J_3$ as follows.

\begin{equation}\nonumber
  \begin{split}
    J_3
    \leq& \; \frac{d}{dt}\int_{\mathbb{R}^2} b_2 |\p_1^2 b_2|^2 \; dx - 3\frac{d}{dt}\int_{\mathbb{R}^2} |b_2|^2 |\p_1^2 b_2|^2 \;dx \\
    &+ C (\|b\|_{H^2} + \|b\|_{H^2}^2 + \|b\|_{H^2}^3) \|\p_2 b\|_{H^1} \|\nabla u\|_{H^2}.
  \end{split}
\end{equation}
Namely,
\begin{equation}\label{J3}
  \int_0^t J_3 \; d\tau \leq C (E_0^\frac{3}{2} + E_0^2) + C (E_0^\frac{1}{2}+ E_0 + E_0^\frac{3}{2})(E_0+ E_1).
\end{equation}
Taking all estimates for $J_1$ through $J_3$ into account, namely \eqref{J1}, \eqref{J2} and \eqref{J3}, we then complete the proof of this Proposition.

\end{proof}

\section{Energy estimates}\label{energy-est}

In this section, we are going to show the energy estimate for $A_{\pm}$, $E_0$ and $E_1$.
The propositions we proved in Section 2 will play an important role. We first prove the following lemma, which offers 
 the $ a \; priori$ estimates for $A_{\pm}$.

\begin{lemma}\label{lem0}
For $A_{mp}$, $E_0$ and $E_1$ defined in \eqref{energy}, it holds  for any $t >0$  that
  \begin{equation}\nonumber
  \begin{split}
    A_{\pm}(t) =& \int_0^t \int_{\mathbb{R}^2} |\Omega_{\pm}|^2 |\p_1^2 b_2|^2 \; dx \; d\tau \\
    \leq& \; C \Big\{(E_0 + E_0^\frac{3}{2})(E_0 + E_1) +   E_0^{\frac{7}{4}}(E_0^{\frac{3}{4}} + E_1^{\frac{3}{4}})+  E_0^\frac{1}{2} A_{\pm} + E_0 A_{\pm}^\frac{1}{2} \Big\}.
  \end{split}
  \end{equation}
\end{lemma}
\begin{proof}

Notice that 
\begin{equation}\label{lem0-eq0}
\begin{split}
&\int_0^t \int_{\mathbb{R}^2} |\Omega_{\pm}|^2  |\p_1^2 b_2|^2 \; dx \; d\tau \\
= & \; C \int_0^t\int_{\mathbb{R}^2} (|\Omega_{\pm}|^2 \mp \frac{1}{2}\frac{d}{dt}\int_{-\infty}^{x_2} |\Omega_{\pm}|^2 \; dy) |\p_1^2 b_2|^2 \; dx \; d\tau\\
& \pm C \int_0^t\int_{\mathbb{R}^2}  \Big(\frac{1}{2}\frac{d}{dt}\int_{-\infty}^{x_2} |\Omega_{\pm}|^2 \; dy \Big)|\p_1^2 b_2|^2 \; dx\; d\tau.
\end{split}
\end{equation}
With the help of  Proposition \ref{prop0}, the first term on the right hand side of \eqref{lem0-eq0} can be controlled easily by
\begin{equation}\label{3.3}
C(E_0 + E_0^\frac{3}{2})(E_0 + E_1) + CE_0^\frac{1}{2} A_{\pm} + C E_0 A_{\pm}^\frac{1}{2}.
\end{equation}
Now, let us focus on the second term on the right hand side of \eqref{lem0-eq0}. Using integration by parts, it holds that 
\begin{align}
    &\mp\int_0^t\int_{\mathbb{R}^2} \Big( \frac{1}{2}\frac{d}{dt}\int_{-\infty}^{x_2} |\Omega_{\pm}|^2 \; dy \Big)|\p_1^2 b_2|^2 \; dx\; d\tau \nonumber \\
    =& \mp\int_0^t\frac{d}{dt}\int_{\mathbb{R}^2} \Big( \frac{1}{2}\int_{-\infty}^{x_2} |\Omega_{\pm}|^2 \; dy \Big)|\p_1^2 b_2|^2 \; dx\; d\tau \nonumber\\
    &\pm\int_0^t \int_{\mathbb{R}^2}  \frac{1}{2}\int_{-\infty}^{x_2} |\Omega_{\pm}|^2 \; dy \frac{d}{dt}|\p_1^2 b_2|^2 \; dx\; d\tau \nonumber\\
    \leq & \; C\big(\|\Omega_{\pm}(t, \cdot)\|_{L_{x_1}^\infty L_{x_2}^2}^2 \|\p_1^2 b_2(t, \cdot)\|_{L^2}^2 + \|\Omega_{\pm}(0, \cdot)\|_{L_{x_1}^\infty L_{x_2}^2}^2 \|\p_1^2 b_2(0, \cdot)\|_{L^2}^2\big)\nonumber\\
    & \pm \int_0^t\int_{\mathbb{R}^2}  \int_{-\infty}^{x_2} |\Omega_{\pm}|^2 \; dy \p_1^2 (\p_2 u_2 + b \cdot \nabla u_2 - u \cdot \nabla b_2) \p_1^2 b_2 \; dx\; d\tau. \nonumber
\end{align}
Which can be furthermore  bounded by the following terms

\begin{align}
 & C \sup_{0 \leq \tau \leq t} (\|u_1\|_{H^2}^2 + \|b_1\|_{H^2}^2) \|b\|_{H^2}^2 \nonumber\\
    &+ C \int_0^t\|\Omega_{\pm}(t, \cdot)\|_{L_{x_1}^\infty L_{x_2}^2}^2 (\|\p_1^2 \p_2 u_2\|_{L^2} + \|b\|_{H^2}\|\nabla u\|_{H^2}) \| \p_1^2 b_2\|_{L^2} d\tau \nonumber\\
    & -C \int_0^t\int_{\mathbb{R}^2} \int_{-\infty}^{x_2} |\Omega_{\pm}|^2 \; dy \; u \cdot \nabla \p_1^2 b_2 \p_1^2 b_2 \; dx\; d\tau.\nonumber
\end{align}
Therefore, applying H\"{o}lder's inequality, one has
\begin{align}
  &\mp\int_0^t\int_{\mathbb{R}^2}  \Big(\frac{1}{2}\frac{d}{dt}\int_{-\infty}^{x_2} |\Omega_{\pm}|^2 \; dy \Big) |\p_1^2 b_2|^2 \; dx\; d\tau \nonumber \\
    \leq & \; C E_0(t)^2 + \sup_{0 \leq \tau \leq t} \|\Omega_{\pm}\|_{L^2}\|\p_1^2 b_2\|_{L^2} \int_0^t\|\p_1 \Omega_{\pm}\|_{L^2}\|\p_1^2 \p_2 u_2\|_{L^2}\; d\tau \nonumber\\
     & + C \sup_{0 \leq \tau \leq t} \|\Omega_{\pm}\|_{L^2}\|b\|_{H^2}\|\p_1^2 b_2\|_{L^2} \int_0^t\|\p_1 \Omega_{\pm}\|_{L^2}\|\nabla u\|_{H^2} d\tau \nonumber\\
     & \mp \int_0^t\int_{\mathbb{R}^2} u_1 \Big( \int_{-\infty}^{x_2}\p_1 \Omega_{\pm} \Omega_{\pm}\; dy \Big)|\p_1^2 b_2|^2 \; dx\; d\tau \nonumber\\
     & \mp \int_0^t\int_{\mathbb{R}^2}  \frac{1}{2} u_2  |\Omega_{\pm}|^2 |\p_1^2 b_2|^2 \; dx\; d\tau. \nonumber
\end{align}
We then obtain the estimate for the last line in \eqref{lem0-eq0} as follows.
\begin{equation}\label{3.4}
\begin{split}
&\mp\int_0^t\int_{\mathbb{R}^2}  \Big(\frac{1}{2}\frac{d}{dt}\int_{-\infty}^{x_2} |\Omega_{\pm}|^2 \; dy \Big) |\p_1^2 b_2|^2 \; dx\; d\tau  \\
    \leq & \; C E_0(t)^2 + C(E_0 + E_0^\frac{3}{2})(E_0 + E_1) \\
     &+ C \int_0^t \|u_1\|_{L^\infty} \|\p_1 \Omega_{\pm}\|_{L_{x_1}^\infty L_{x_2}^2} \|\Omega_{\pm}\|_{L_{x_1}^\infty L_{x_2}^2} \|\p_1^2 b_2\|_{L^2}^2 \; dx\; d\tau \\
     &+ C\sup_{0\leq \tau \leq t} \|u_2\|_{L^\infty} \int_0^t\int_{\mathbb{R}^2}  |\Omega_{\pm}|^2 |\p_1^2 b_2|^2 \; dx\; d\tau \\
    \leq & \; C E_0(t)^2 + C(E_0 + E_0^\frac{3}{2}) (E_0  + E_1) + CE_0^{\frac{7}{4}}(E_0^{\frac{3}{4}}+E_1^{\frac{3}{4}})    + C E_0^\frac{1}{2} A_{\pm}(t).
\end{split}
\end{equation}

Combing \eqref{3.3} and \eqref{3.4} together, we then finish the proof of Lemma 3.1 and get the energy estimate for $A_{\pm}(t)$.
\end{proof}

The next lemma is designed to give the $ a \; priori$ estimate for basic energy $E_0(t)$.

\begin{lemma}
For $A_{mp}$, $E_0$ and $E_1$ defined in \eqref{energy}, it holds that

\begin{equation}\nonumber
\begin{split}
& E_0(t) = \sup_{0 \leq \tau \leq t} (\|u\|_{H^2}^2 + \|b\|_{H^2}^2) + \int_0^t \|\nabla u\|_{H^2}^2 \; d\tau\\
 \leq& \; CE_0(0) + CE_0E_1^\frac{1}{2} +  CE_0 ( A_+ + A_-)^\frac{1}{2} \\
 & + C (E_0^\frac{1}{2}+ E_0 + E_0^\frac{3}{2})(E_0 + E_1).
\end{split}
\end{equation}
\end{lemma}
\begin{proof}
Due to the norm equivalence, namely
$$\|f\|_{H^2}^2 \sim \|f\|_{L^2}^2 + \sum_{i = 1}^2 \|\p_i^2 f\|_{L^2}^2,$$
it suffices to bound
$$\|(u, b)\|_{L^2} + \sum_{i = 1}^2 \|\p_i^2 (u, b)\|_{L^2}$$
instead of $E_0(t)$ itself.
Therefore, we can split the proof into two steps as follows.

\noindent $\bullet$ $L^2$ estimate \\
This is  the standard $L^2$ energy estimate for MHD systems. Multiplying the first equation in \eqref{mhd0} by $u$ and the second equation by $b$, then integrating over $\mathbb{R}^2$, one has 
\begin{equation}\nonumber
  \frac{1}{2}\frac{d}{dt} (\|u\|_{L^2}^2 + \|b\|_{L^2}^2)+ \|\nabla u\|_{L^2}^2 = 0.
\end{equation}

\noindent $\bullet$ Higher order norms estimate\\
It's easy to derive the following equation from the MHD equations \eqref{mhd0},
\begin{equation}\nonumber
\begin{split}
&\frac{1}{2}\frac{d}{dt} \sum_{i = 1}^2 (\|\p_i^2 u\|_{L^2}^2 + \|\p_i^2 b\|_{L^2}^2)  + \sum_{i = 1}^2 \|\p_i^2 \nabla u\|_{L^2}^2 \\
= & \; M_1 +M_2+M_3+M_4+M_5,
\end{split}
\end{equation}
where,
\begin{equation}\nonumber
  \begin{split}
    M_1 &= \sum_{i = 1}^2 <\p_i^2 \p_2 b, \p_i^2 u> + \sum_{i = 1}^2 <\p_i^2 \p_2 u, \p_i^2 b>,\\
    M_2 &= - \sum_{i = 1}^2 <\p_i^2 (u \cdot \nabla u), \p_i^2 u>,\\
    M_3 &= \sum_{i = 1}^2 <\p_i^2 (b \cdot \nabla b) - b \cdot \nabla \p_i^2 b, \p_i^2 u>, \\
    M_4 &= \sum_{i = 1}^2 <\p_i^2 (b \cdot \nabla u) - b \cdot \nabla \p_i^2 u, \p_i^2 b>, \\
    M_5 &= - \sum_{i = 1}^2 <\p_i^2 (u \cdot \nabla b), \p_i^2 b>.
  \end{split}
\end{equation}
For the first term $M_1$, integration by parts will lead to
\begin{equation}\label{m1}
  M_1 = \sum_{i = 1}^2 <\p_i^2 \p_2 b, \p_i^2 u> - \sum_{i = 1}^2 <\p_i^2 u, \p_i^2 \p_2 b> \; = 0.
\end{equation}
For the second term $M_2$, it's also trivial to get
\begin{equation}\label{m2}
M_2 \leq \; C \|u\|_{H^2} \|\nabla u\|_{H^2}^2.
\end{equation}
Using $\p_1 b_1 = - \p_2 b_2$, the third term $M_3$ can be written into
\begin{equation}\nonumber
\begin{split}
M_3=& <\p_2^2 (b \cdot \nabla b) - b \cdot \nabla \p_2^2 b, \p_2^2 u> + <\p_1^2 b \cdot \nabla b, \p_1^2 u> \\
& \qquad \qquad + 2<\p_1 b \cdot \nabla \p_1 b, \p_1^2 u>,
\end{split}
\end{equation}
which will be bounded furthermore by
\begin{equation}\nonumber
\begin{split}
 \; C\|b\|_{H^2}\|\p_2 b\|_{H^1}&\|\nabla u\|_{H^2} + <\p_1^2 b_1 \p_1 b, \p_1^2 u> + 2<\p_1 b_1 \p_1^2 b, \p_1^2 u>\\
 & + <\p_1^2 b_2 \p_2 b, \p_1^2 u> + 2<\p_1 b_2 \p_2 \p_1 b, \p_1^2 u>.
\end{split}
\end{equation}
Consequently, one has
\begin{equation}\label{m3}
 M_3 \leq  \; C\|b\|_{H^2}\|\p_2 b\|_{H^1}\|\nabla u\|_{H^2}.
\end{equation}
The estimate for $M_4$ needs more efforts. It's standard to obtain
\begin{equation}\nonumber
\begin{split}
M_4
=& <\p_2^2 (b \cdot \nabla u) - b \cdot \nabla \p_2^2 u, \p_2^2 b> + <\p_1^2 (b \cdot \nabla u) - b \cdot \nabla \p_1^2 u, \p_1^2 b> \\
\leq& \; C\|b\|_{H^2}\|\p_2 b\|_{H^1}\|\nabla u\|_{H^2} + <\p_1^2 b_1 \p_1 u + 2 \p_1 b_1 \p_1^2 u, \p_1^2 b> \\
& +  <\p_1^2 b_2 \p_2  u_1 + 2 \p_1 b_2 \p_2 \p_1 u_1, \p_1^2 b_1> \\
&+ <\p_1^2 b_2 \p_2  u_2 + 2 \p_1 b_2 \p_2 \p_1 u_2, \p_1^2 b_2>.
\end{split}
\end{equation}
Moreover, it can be bounded by
\begin{equation}\nonumber
\begin{split}
& \; C\|b\|_{H^2}\|\p_2 b\|_{H^1}\|\nabla u\|_{H^2} + <\p_1^2 b_2  \p_2 u_2, \p_1^2 b_2> + 2<\p_1^2 u_2, \p_1 b_2 \p_2 \p_1 b_2> \\
\leq& \; C\|b\|_{H^2}\|\p_2 b\|_{H^1}\|\nabla u\|_{H^2} + <\p_2 u_2, |\p_1^2 b_2|^2>.
\end{split}
\end{equation}
Notice here we have trouble when deriving the decay information for $\|\p_1^2 b_2\|_{L^2}$, it's then hard for us to give the estimate for the wildest term $<\p_2 u_2, |\p_1^2 b_2|^2>$ in above inequality.
Fortunately, we succeed in transferring this term to some other good terms using the inherent structure of system concerned.
Applying Proposition \ref{prop1}, it's easy to obtain
\begin{equation}\label{m4}
\begin{split}
  \int_0^t M_4 \; d\tau \leq &\; CE_0E_1^\frac{1}{2} +  CE_0 ( A_+ + A_-)^\frac{1}{2}\\
   &+ C (E_0^\frac{1}{2}+ E_0 + E_0^\frac{3}{2})(E_0 + E_1).
\end{split}
\end{equation}
Now, we turn to deal with the last term $M_5$. Using the divergence free condition $\nabla \cdot u = \nabla \cdot b = 0$, it's obvious that
\begin{equation}\nonumber
  \begin{split}
  M_5 = & \; -  <\p_2^2 (u \cdot \nabla b), \p_2^2 b> - <\p_1^2 (u \cdot \nabla b_1) , \p_1^2 b_1> \\
 & - <\p_1^2 (u \cdot \nabla b_2) , \p_1^2 b_2>,
\end{split}
\end{equation}
which can be bounded by
\begin{equation}\nonumber
\begin{split}
    & \; C \|b\|_{H^2} \|\p_2 b\|_{H^1} \|\nabla u\|_{H^2} - <\p_1^2 u_2 \p_2 b_2 , \p_1^2 b_2> - 2 <\p_1 u_2 \p_1 \p_2 b_2 , \p_1^2 b_2> \\
    & \qquad \qquad - <\p_1^2 u_1 \p_1 b_2 , \p_1^2 b_2> - 2 <\p_1 u_1 \p_1^2 b_2 , \p_1^2 b_2> \\
    \leq& \; C \|b\|_{H^2} \|\p_2 b\|_{H^1} \|\nabla u\|_{H^2} + <\p_1 \p_2 u_2  , \frac{1}{2}\p_1 |\p_1 b_2|^2> + 2 <\p_2 u_2 , |\p_1^2 b_2|^2> \\
    \leq& \;C \|b\|_{H^2} \|\p_2 b\|_{H^1} \|\nabla u\|_{H^2} + \frac{1}{2} <\p_1^2 u_2  , \p_2|\p_1 b_2|^2> + 2 <\p_2 u_2 , |\p_1^2 b_2|^2> \\
    \leq& \; C \|b\|_{H^2} \|\p_2 b\|_{H^1} \|\nabla u\|_{H^2} + 2 <\p_2 u_2 , |\p_1^2 b_2|^2>.
  \end{split}
\end{equation}
Applying Proposition \ref{prop1} again, we achieved the following time integral estimate for $M_5$ 
\begin{equation}\label{m5}
  \begin{split}
    \int_0^t M_5 \; d\tau \leq & \; CE_0E_1^\frac{1}{2} +  CE_0 ( A_+ + A_-)^\frac{1}{2} \\
    & + C (E_0^\frac{1}{2}+ E_0 + E_0^\frac{3}{2})(E_0 + E_1).
  \end{split}
\end{equation}
Combining all the estimate for $M_1 \thicksim M_5$ namely \eqref{m1}, \eqref{m2},  \eqref{m3}, \eqref{m4} and \eqref{m5},, we complete the proof of this Lemma.

\end{proof}

The next lemma provides the decay information for the key term $\|\p_2 b\|_{H^1}$.
\begin{lemma}
For $A_{mp}$, $E_0$ and $E_1$ defined in \eqref{energy}, it holds that
  \begin{equation}\nonumber
  \begin{split}
    E_1(t) = & \int_0^t \|\p_2 b\|_{H^1}^2 \; d\tau  \\
    \leq & \; C \Big\{ E_0 + E_0^\frac{1}{2}E_1^\frac{1}{2} + E_0^\frac{1}{2}(E_0 + E_1) + (A_+^\frac{1}{2} + A_-^\frac{1}{2})E_1^\frac{1}{2} \Big\}.
    \end{split}
  \end{equation}
\end{lemma}
\begin{proof}
Using the following identify
\begin{equation}\label{p2b}
  \p_2 b = u_t + u \cdot \nabla u - \Delta u + \nabla P - b \cdot \nabla b ,
\end{equation}
we split the estimate for $E_1$ into lower order terms part and higher order terms part as before.\\
$\bullet$ Lower order norms estimate \\
By \eqref{p2b} we can write
\begin{equation}\nonumber
\begin{split}
  \|\p_2 b\|_{L^2}^2  =&  < u_t + u \cdot \nabla u - \Delta u + \nabla P - b \cdot \nabla b,  \p_2 b>\\
   = & \; Q_1 +Q_2 +Q_3,
\end{split}
\end{equation}
where
\begin{equation}\nonumber
\begin{split}
  Q_1 =& < u_t,  \p_2 b>,\\
  Q_2 =&<u \cdot \nabla u - \Delta u + \nabla P,  \p_2 b>, \\
  Q_3 =& \; - <  b \cdot \nabla b,  \p_2 b>.
\end{split}
\end{equation}
The estimate process for $Q_1$, $Q_2$ and $Q_3$ are almost the same like before. Therefore, we give the following estimates directly.
\begin{equation}\label{q1}
\begin{split}
  Q_1 =& \; \frac{d}{dt} < u,  \p_2 b> +  < u, \p_2 (-u \cdot \nabla b + b \cdot \nabla u + \p_2 u)> \\
  \leq &  \; \frac{d}{dt} <u, \p_2 b> + C \|\nabla u\|_{L^2}^2 \\
  & + C (\|u\|_{H^2} + \|b\|_{H^2})(\|\p_2 b\|_{H^1}^2 + \|\nabla u\|_{H^2}^2),
\end{split}
\end{equation}
and,
\begin{equation}\label{q2}
\begin{split}
  Q_2 \leq \; C(\|\nabla u\|_{H^1} + \|u\|_{H^2} \|\nabla u\|_{H^2}) \|\p_2 b\|_{L^2},
\end{split}
\end{equation}
moreover,
\begin{equation}\label{q3}
\begin{split}
  Q_3 =& \; -< ( b_2 \p_2 b) ,  \p_2 b> -  < ( b_1 \p_1 b) ,  \p_2 b> \\
  \leq & \; \|b\|_{H^2} \|\p_2 b\|_{L^2}^2 + \|b_1\|_{L^2}^\frac{1}{2}\|\p_1 b_1\|_{L^2}^\frac{1}{2} \|\p_1 b\|_{L^2}^\frac{1}{2}\|\p_2 \p_1 b\|_{L^2}^\frac{1}{2}\|\p_2 b\|_{L^2} \\
  \leq & \; C \|b\|_{H^2} \|\p_2 b\|_{H^1}^2.
\end{split}
\end{equation}
$\bullet$ Higher order norms estimate \\
Similarly, we can write
\begin{equation}\nonumber
\begin{split}
  \sum_{i = 1}^2 \|\p_i \p_2 b\|_{L^2}^2 =& \sum_{i = 1}^2 <\p_i (u_t + u \cdot \nabla u - \Delta u + \nabla P - b \cdot \nabla b) , \p_i \p_2 b>\\
   = & \; Q_4 + Q_5 +Q_6,
\end{split}
\end{equation}
where
\begin{equation}\nonumber
\begin{split}
  Q_4 =& \sum_{i = 1}^2 <\p_i u_t, \p_i \p_2 b>,\\
  Q_5 =& \sum_{i = 1}^2 <\p_i (u \cdot \nabla u - \Delta u + \nabla P), \p_i \p_2 b>, \\
  Q_6 =& -\sum_{i = 1}^2 <\p_i ( b \cdot \nabla b) , \p_i \p_2 b>.
\end{split}
\end{equation}
The estimate for $Q_4$ and $Q_5$ is trivial, we can directly get
\begin{equation}\label{q4}
\begin{split}
  Q_4 =& \sum_{i = 1}^2 \frac{d}{dt} <\p_i u, \p_i \p_2 b> + \sum_{i = 1}^2 <\p_i^2 u, \p_2 (-u \cdot \nabla b + b \cdot \nabla u + \p_2 u)> \\
  \leq & \sum_{i = 1}^2 \frac{d}{dt} <\p_i u, \p_i \p_2 b> + C \|\nabla u\|_{H^2}^2 \\
   & + C (\|u\|_{H^2} + \|b\|_{H^2})(\|\p_2 b\|_{H^1}^2 + \|\nabla u\|_{H^2}^2),
\end{split}
\end{equation}
and,
\begin{equation}\label{q5}
\begin{split}
  Q_5 \leq C(\|\nabla u\|_{H^2} + \|u\|_{H^2} \|\nabla u\|_{H^2}) \|\p_2 b\|_{H^1}.
\end{split}
\end{equation}
While for the last term $Q_6$, we divide it into five parts
\begin{equation}\nonumber
\begin{split}
  Q_6 =& -\sum_{i = 1}^2 <\p_i ( b_2 \p_2 b) , \p_i \p_2 b> -  <\p_2 ( b_1 \p_1 b) , \p_2 \p_2 b> \\
  & - <\p_1 b_1 \p_1 b , \p_1 \p_2 b> - < b_1 \p_1^2 b_1 , \p_1 \p_2 b_1>\\
  &  - <b_1 \p_1^2 b_2 , \p_1 \p_2 b_2>.
\end{split}
\end{equation}
The first four terms in above equality can be easily bounded by $\|b\|_{H^2} \|\p_2 b\|_{H^1}^2$. And for the last term, applying H\"{o}lder's inequality yields that
\begin{equation}\label{q6}
  \begin{split}
   & \int_0^t \int_{\mathbb{R}^2} b_1 \p_1^2 b_2 \p_1 \p_2 b_2 \; dx \; d\tau \\
    \leq & \; \big(\int_0^t \int_{\mathbb{R}^2} |b_1|^2 |\p_1^2 b_2|^2 \; dx \;d\tau\big)^\frac{1}{2} \big(\int_0^t \int_{\mathbb{R}^2} |\p_1\p_2 b_2|^2  \; dx \;d\tau\big)^\frac{1}{2} \\
    \leq & \; (A_+^\frac{1}{2} + A_-^\frac{1}{2})E_1^\frac{1}{2}.
  \end{split}
\end{equation}
Combining the estimate for  $Q_1 \sim Q_6$ together, namely \eqref{q1}, \eqref{q2}, \eqref{q3}, \eqref{q4}, \eqref{q5} and \eqref{q6}, notice the following equivalence
$$\|\p_2 b\|_{H^1}^2 \sim \|\p_2 b\|_{L^2}^2 + \sum_{i = 1}^2 \|\p_i \p_2 b\|_{L^2}^2,$$
we then obtain
\begin{equation}\nonumber
  \begin{split}
    \int_0^t \|\p_2 b\|_{H^1}^2 \; d\tau \leq & C \big\{ E_0(t) + E_0^\frac{1}{2}E_1^\frac{1}{2} + E_0^\frac{1}{2}(E_0 + E_1) + (A_+^\frac{1}{2} + A_-^\frac{1}{2})E_1^\frac{1}{2} \big\}.
  \end{split}
\end{equation}
It then finishes the proof of this lemma.
\end{proof}
\vskip .3in
\section{The proof of theorem \ref{thm}}

This section completes the proof of Theorem \ref{thm}, which can be achieved by applying the bootstrapping argument
to the energy inequalities obtained in the previous lemmas.

As aforementioned in the introduction, the local well-posedness theorem
of system (\ref{mhd}) in $H^2$ space follows from a standard procedure (see, e.g., \cite{MaBe}). Therefore our
attention is exclusively focused on the global bound {of $H^2$-norms}. This is accomplished
by the bootstrapping argument. The key components are the energy inequalities established previously
in Sections \ref{energy-est},
  \begin{equation}\nonumber
  \begin{split}
    A_{\pm}(t) \leq& \; C \big( E_0(t) + E_0^\frac{3}{2}(t)\big)\big(E_0(t) + E_1(t)\big) + C E_0^\frac{1}{2}(t) A_{\pm}(t) + C E_0(t) A_{\pm}^\frac{1}{2}(t), \\
    E_0(t) \leq& \; CE_0(0) + CE_0(t)E_1^\frac{1}{2}(t) +  CE_0(t) \big( A_+(t) + A_-(t)\big)^\frac{1}{2} \\
    &\;+ C \big(E_0^\frac{1}{2}(t)+ E_0(t) + E_0^\frac{3}{2}(t)\big)\big(E_0(t) + E_1(t)\big),\\
     E_1(t) \leq& C E_0(t) + C E_0^\frac{1}{2}(t)E_1^\frac{1}{2}(t) + C E_0^\frac{1}{2}(t)\big(E_0(t) + E_1(t)\big) \\
     & + C\big(A_+^\frac{1}{2}(t)+ A_-^\frac{1}{2}(t)\big)E_1^\frac{1}{2}(t).
  \end{split}
  \end{equation}
Applying Young's inequality, there exists a positive constant $C_0 \ge 1$ and for any $t>0$ there holds
\begin{equation}\label{en}
		{\mathscr{E} (t)}\le C_0\, \mathscr{E}(0) +C_0 \,\mathscr{E}^{\frac32}(t)
		+ C_0\, \mathscr{E}^2 (t)+ C_0 \,\mathscr{E}^{\frac52}(t),
\end{equation}
where $\mathscr{E}(t) = E_0(t) + E_1(t) + A_-(t) + A_+(t)$ was defined in section 1.  Applying the bootstrapping argument to (\ref{en}) implies that, if $\|(u_0, b_0)\|_{H^2}$ is
sufficiently small, say
\begin{equation}\nonumber
\mathscr{E}(0) \le \frac1{128 C^3_0}\quad \mbox{or}\quad  \|(u_0, b_0)\|_{H^2} \le \epsilon :=\frac1{\sqrt{128 C_0^3}},
\end{equation}
then for any $t>0$ we have
$$
	\mathscr{E}(t)\le \frac1{32 C^2_0}, \quad \mbox{especially}\quad \|(u(t), b(t))\|_{H^2} \le 2 \sqrt{C_0}\, \epsilon.
$$
Here we have used the initial data assumption \eqref{initial}.
In fact, if we make the ansatz that
\begin{equation}\label{ans}
\mathscr{E}(t)\le \frac1{16 C^2_0}
\end{equation}
and insert (\ref{ans}) in (\ref{en}), we  shall finally obtain
$$
\mathscr{E}(t) \le \frac{1}{128C_0^2} + \frac{1}{64C_0^2} + \frac{1}{256C_0^2} + \frac{1}{1024C_0^2} \le \frac{1}{32 C^2_0}.
$$
This completes the proof of Theorem \ref{thm}.

\vskip .3in
\centerline{\bf Acknowledgments}

\vskip .1in
Shijin Ding is supported by the Key Project of National Natural Science Foundation of China (No. 12131010), the National Natural Science Foundation of China (No. 12271032).
Ronghua Pan is supported by National Science Foundation (No. DMS-2108453).
Yi Zhu is supported by the National Natural Science Foundation of China (No. 12271160).

\end{document}